\documentclass{ws-ijnt}

\usepackage{latexsym}
\usepackage{amsmath}
\usepackage{amssymb}
\usepackage{amsfonts}
\usepackage{amscd}
\usepackage{mathrsfs}

\usepackage[english]{babel}

\newcommand{\Q}{\mathbb{Q}}
\newcommand{\Z}{\mathbb{Z}}

\newcommand{\C}{\mathbb{C}}
\newcommand{\F}{\mathbb{F}}

\newcommand{\zzeta}{\overline{\zeta_3}}
\DeclareMathOperator{\car}{char}
\DeclareMathOperator{\tr}{tr}

\DeclareMathOperator{\Frob}{Frob}
\DeclareMathOperator{\Gal}{Gal}
\DeclareMathOperator{\SL}{SL}
\DeclareMathOperator{\GL}{GL}
\DeclareMathOperator{\Aut}{Aut}

\newcommand{\mat}[4]{\left(\begin{array}{cc}
#1 &#2\\
#3 &#4
\end{array}\right)}

\usepackage{tikz}
\usetikzlibrary{cd}
\usetikzlibrary{shapes.misc}

\begin{document}

\markboth{Nirvana Coppola}{Wild Galois representations: $2$-adic fields}

%
\catchline{}{}{}{}{}
%

\title{Wild Galois representations: elliptic curves over a $2$-adic field with non-abelian inertia action}

\author{Nirvana Coppola\footnote{
University of Bristol}}

\address{School of Mathematics, University of Bristol, Fry Building,\\
University Walk, Bristol, BS8 1UG, United Kingdom\\
\email{nc17051@bristol.ac.uk} }

\maketitle

\begin{history}
\received{}
\accepted{}
\end{history}

\begin{abstract}
In this paper we present a description of the $\ell$-adic Galois representation attached to an elliptic curve defined over a $2$-adic field $K$, in the case where the image of inertia is non-abelian. There are two possibilities for the image of inertia, namely $Q_8$ and $\SL_2(\F_3)$, and in each case we need to distinguish whether the inertia degree of $K$ over $\Q_2$ is even or odd. The result presented here are being implemented in an algorithm to compute explicitly the Galois representation in these four cases.
\end{abstract}



\begin{section}{Introduction}
Let $K$ be a $2$-adic field (i.e. a finite extension of $\Q_2$ or, equivalently, a non-archimedean local field with characteristic $0$ and residue characteristic $2$) and let $E/K$ be an elliptic curve. Since $\car (K)=0$ we can always assume that $E$ is in short Weierstrass form, $E: y^2=x^3+a_4x+a_6$, for $a_4,a_6 \in K$. Let $k$ be the residue field of $K$ and let $n=[k:\F_2]$, i.e. $n$ is the inertia degree of $K$. Suppose that $E/K$ has potential good reduction, that is it has additive reduction and there exists a finite extension of $K$ where $E$ acquires good reduction.

Let $\overline{K}$ be a fixed algebraic closure of $K$ and let $G_K=\Gal(\overline{K}/K)$ be the absolute Galois group of $K$, which acts on the points of $E(\overline{K})$. This induces a representation $\rho$ on the $\ell$-adic Tate module $T_\ell (E)$, which is independent of the prime $\ell$ as long as $\ell \neq 2$, in the sense of \cite[\S 2 Theorem 2.ii]{10.2307/1970722}.

More precisely we will denote by $\rho_\ell$ or simply $\rho$ the following representation:
\begin{center}
    \begin{tikzcd}
    G_K \arrow{r}& Aut(T_\ell (E) \otimes \overline{\Q}_\ell),
    \end{tikzcd}
\end{center}
which is a $2$-dimensional representation over $\overline{\Q}_\ell$, so after fixing a basis for $T_\ell (E) \otimes \overline{\Q}_\ell$ we can identify $Aut(T_\ell (E) \otimes \overline{\Q}_\ell)$ with $\GL_2(\overline{\Q}_\ell)$. Let us consider the restriction of $\rho$ to the inertia subgroup $I_K \cong \Gal(\overline{K} / K^{nr})$ of $G_K$ (here $K^{nr}$ is the maximal unramified extension of $K$). If $L$ is the minimal extension of $K^{nr}$ where $E$ acquires good reduction, which exists by \cite[\S 2 Corollary 3]{10.2307/1970722}, then $\ker (\rho)= \Gal(\overline{K}/L)$ and the image of inertia is isomorphic to $\Gal(L/K^{nr})$, as a consequence of the Criterion of N\'{e}ron-Ogg-Shafarevich (see \cite[VII \S 7 Theorem 7.1]{silverman1986arithmetic}).  Moreover, it is proved in \cite[Theorems 2,3]{Kraus1990} that the image of inertia can only be one of the following:
\begin{align}
    C_2, C_3, C_4, C_6, Q_8, \SL_2(\F_3).
\end{align}

For a more explicit approach, see also \cite[Part IV \S 11,12]{krausfreitas2016}.
In this paper we focus on the cases where $I$ is non-abelian (equivalently non-cyclic), hence it is either $Q_8$ or $\SL_2(\F_3)$. In \cite[Theorem 3]{Kraus1990}, there is a criterion to check whether this holds.

Recall that the quotient $G_K/I_K$ is isomorphic to the absolute Galois group of the residue field, which is pro-cyclic and generated by the Frobenius element, that acts as $x \mapsto x^q$ where $q=|k|$. We call an Arithmetic Frobenius of $K$, and denote it by $\Frob_K$, any fixed choice of an element of $G_K$ that reduces to the Frobenius element modulo $I_K$. In order to compute explicitly the elements in the image of $\rho$, let us fix an embedding $\overline{\Q}_\ell \rightarrow \C$; in particular we will identify the element $\sqrt{-2}$ of $\overline{\Q}_\ell$ with $i \sqrt{2} \in \C$.

We will prove the following result. We refer to \cite{groupnames} for the notation used for group names and character tables; in particular we denote each conjugacy class by the order of its elements, followed by a letter if there is more than one class with the same order.
\begin{theorem}\label{mainthm}
Let $E/K$ be an elliptic curve with potential good reduction over a $2$-adic field, let $\ell$ be a prime different from $2$ and let $\rho: G_K \rightarrow \GL_2(\overline{\Q}_\ell)$ be the $\ell$-adic Galois representation attached to $E$. Suppose that $I= \rho(I_K)$ is non-abelian. Let $\Delta$ be the discriminant of a (not necessarily minimal) equation for $E$ and let $n$ be the inertia degree of $K/\Q_2$. Then $\rho$ factors as
\begin{align}
    \rho = \chi \otimes \psi
\end{align}
where $\chi : G_K \rightarrow \overline{\Q}_\ell^\times$ is the unramified character mapping the (Arithmetic) Frobenius of $K$ to $\sqrt{-2}^n$, and $\psi$ is the irreducible $2$-dimensional representation of the group $G=\Gal(K(E[3])/K)$ given as follows.
\begin{itemize}
    \item If $n$ is even and $\Delta$ is a cube in $K$, then $\psi$ is the representation of $G=Q_8$ with character
$$
\begin{array}{c|rrrrr}
  \rm class&\rm1&\rm2&\rm4A&\rm4B&\rm4C\cr
  \rm size&1&1&2&2&2\cr
\hline
  \psi&2&-2&0&0&0\cr
\end{array}
$$
    \item If $n$ is even and $\Delta$ is not a cube in $K$, then $\psi$ is the representation of $G=\SL_2(\F_3)$ with character
$$
\begin{array}{c|rrrrrrr}
  \rm class&\rm1&\rm2&\rm3A&\rm3B&\rm4&\rm6A&\rm6B\cr
  \rm size&1&1&4&4&6&4&4\cr
\hline
  \psi&2&-2&-1&-1&0&1&1\cr
\end{array}
$$
Moreover the image of inertia is $Q_8$ if $\Delta$ is a cube in $K^{nr}$ and $\SL_2(\F_3)$ otherwise.
    \item If $n$ is odd and $\Delta$ is a cube in $K$ (equivalently the image of inertia is $Q_8$), then $\psi$ is the representation of $G=SD_{16}$ with character
$$
\begin{array}{c|rrrrrrr}
  \rm class&\rm1&\rm2A&\rm2B&\rm4A&\rm4B&\rm8A&\rm8B\cr
  \rm size&1&1&4&2&4&2&2\cr
\hline
  \psi&2&-2&0&0&0&\sqrt{-2}&-\sqrt{-2}\cr
\end{array}
$$
    \item If $n$ is odd and $\Delta$ is not a cube in $K$ (equivalently the image of inertia is $\SL_2(\F_3)$), then $\psi$ is the representation of $G=\GL_2(\F_3)$ with character
$$
\begin{array}{c|rrrrrrrr}
  \rm class&\rm1&\rm2A&\rm2B&\rm3&\rm4&\rm6&\rm8A&\rm8B\cr
  \rm size&1&1&12&8&6&8&6&6\cr
\hline
  \psi&2&-2&0&-1&0&1&\sqrt{-2}&-\sqrt{-2}\cr
\end{array}
$$
\end{itemize}

In the last two cases a generator for the class $8A$ can be described explicitly (it is $\phi \sigma$ in the proof of Theorem \ref{nodd}).
\end{theorem}

This theorem is almost completely proved in \cite[\S 5]{dokchitser2008root}. In particular the cases where $n$ is even are already known, and here we present a proof for completeness. The cases where $n$ is odd are more subtle. Although it can be easily proved that the representation $\psi$ can only be either the one described above or the one which has the same character values for every conjugacy class except for the classes $8A$ and $8B$, which are swapped, it is not trivial to identify which of these two is equal to $\psi$. In this work we prove that, with the definition of $\chi$ made in the statement of Theorem \ref{mainthm}, only one of the two possible cases occur for elliptic curves. The method of proof consists of describing explicitly a generator of the class $8A$ and computing the trace of $\psi$ on it.


\end{section}

\begin{section}{The good model}\label{sec:good model}

In the following, $E$ is an elliptic curve over a $2$-adic field $K$, with potential good reduction, such that the Galois action attached to it has non-abelian inertia image $I$.

\begin{lemma}\label{good model}
Let $F$ be the field obtained from $K$ by adjoining the coordinates of one point of exact order $3$ and a cube root of the discriminant $\Delta$ of $E$. Then $E$ acquires good reduction over $F$ and it reduces to $\tilde E_F : y^2+y=x^3$ on the residue field.
\end{lemma}

\begin{proof}
Let $P=(x_P,y_P)$ be a non-trivial $3$-torsion point with coordinates in $F$ and let $\lambda_P$ be the slope of the tangent line at $P$. Then after applying the change of coordinates
\begin{align}
    \left\lbrace
    \begin{array}{rl}
         x& \mapsto x+x_P  \\
         y& \mapsto y +\lambda_P x + y_P 
    \end{array}
    \right.
\end{align}
we get an equation for $E$ over $F$ with the same discriminant $\Delta$, of the form
\begin{align}\label{auxiliary equation}
y^2 + A xy + B y = x^3,
\end{align}
with $B \neq 0$ (for a detailed computation see \cite[\S 2, Proposition 2.22 and Corollary 2.23]{Robson2017}).

Next we prove that $\sqrt[3]{B} \in F$. Note that the discriminant of equation \eqref{auxiliary equation} is given by
\begin{align}
    \Delta = -27 B^4 + (AB)^3;
\end{align}
since $\Delta$ and $-B^3$ are cubes in $F$, we have that also $27B-A^3 = 27B \big( 1 - \frac{A^3}{27B} \big)$ is a cube in $F$. If we show that the quantity $1 - \frac{A^3}{27B} $ is a cube in $F$, then also $B$ is. To prove this claim, it is sufficient to show that the valuation in $F$ of $\frac{A^3}{27B}$ is strictly positive. Then we conclude using Hensel's Lemma that the polynomial $z^3 - \big( 1 - \frac{A^3}{27B} \big)$ has a root in $F$.

We know by \cite[\S2 Corollary 2]{10.2307/1970722} that $E$ acquires good reduction over the field $K(E[3])$. We write $v'$ for the normalized valuation on this field, and $v_F$ for the normalized valuation on $F$. As shown in the proof of Theorem 2 in \cite{10.2307/1970722}, the image of inertia under the Galois action injects into $\Aut(\tilde E_{K(E[3])})$, therefore by the classification of the automorphisms of an elliptic curve over a field of characteristic $2$ (see \cite[III \S 10, Theorem 10.1]{silverman1986arithmetic}) it can be non-abelian only if $v'(j) >0$, where $j$ is the $j$-invariant of the curve, and therefore we have $v_F(j) >0$.

Assume by contradiction that $v_F\big(\frac{A^3}{27B}\big) \leq 0$, or equivalently $3v_F(A) \leq v(B)$. By direct computation,
\begin{align}
    j = \dfrac{A^3 (A^3-24B)^3}{B^3 (A^3-27B)}
\end{align}
so we have that the valuation of the numerator is $12 v_F(A)$, and the valuation of the denominator is at least $3v_F(B) + 3v_F(A)$. Now
\begin{align}
    v_F(j) \leq 12 v_F(A) - 3(v_F(B)+v_F(A)) = 3 (3 v_F(A) - v_F(B)) \leq 0,
\end{align}
contradicting the fact that $v_F(j) >0$.

Therefore $B$ is a cube in $F$ and the following is a well-defined change of variables over the field $F$.
\begin{align}
    \left\lbrace
    \begin{array}{ll}
        x &\mapsto \big(\sqrt[3]{B})^2 x \\
        y &\mapsto \big(\sqrt[3]{B})^3 y
    \end{array}
    \right.
\end{align}

After applying this transformation to the curve \eqref{auxiliary equation}, we get the model $y^2 + A' xy + y =x^3$, with $A'=A/\sqrt[3]{B}$. By the computation above, $v_F(A) > v_F(B)/3$, so $v_F(A')>0$ and the valuation of the discriminant is $v_F(-27 B^4 + (AB)^3) - 12 v_F(\sqrt[3]{B})=0$. Therefore this model reduces to $y^2+y=x^3$ on the residue field of $F$, and in particular $E$ acquires good reduction over $F$.
\end{proof}

Computationally it is possible to find the values $x_P,y_P,\lambda_P$ using the following modified version of the $3$-division polynomial, whose roots are precisely the slopes of all tangent lines at the non-trivial $3$-torsion points (for a proof, see \cite[Theorem 1]{dokchitser2007}):
\begin{align}
    \gamma(t)=t^8+18a_4t^4+108a_6t^2-27a_4^2.
\end{align}

If $\lambda_P$ is a root of $\gamma$, then the corresponding point $P$ has coordinates $x_P=\frac{\lambda_P^2}{3}$, $y_P=\frac{\lambda_P^4+3a_4}{6\lambda_P}$.

Let $F^{nr}$ be the maximal unramified extension of $F$, which is equal to the compositum of $F$ and $K^{nr}$. Note that $F^{nr}$ is the minimal extension of $K^{nr}$ where the curve $E$ acquires good reduction. Indeed if $L$ is such extension then by \cite[\S 2 Corollary $2$]{10.2307/1970722}, we have that $L=K^{nr}(E[3])$ and so it clearly contains the coordinates of any $3$-torsion point and any cube root of $\Delta$, which by an easy computation can be expressed in terms of these coordinates, so $F^{nr} \subseteq L$ (see \cite[\S 2 Lemma 2.20]{Robson2017}). On the other hand E does acquire good reduction over $F$, hence on $F^{nr}$, so $L=F^{nr}$ by minimality. Also note that $\ker (\rho)=\Gal(\overline{K}/L)$ and so the representation factors through $\Gal(L/K)$ and the representation induced here is injective.


We have that $[L:K^{nr}] \mid [F:K]$, and since we are assuming that $I$ is non-abelian then $[L:K^{nr}]$ is either $8$ or $24$, so $8 \mid [F:K]$. This occurs precisely when the extension given by adjoining the coordinates of $P$ is totally ramified of degree $8$, i.e. when the polynomial $\gamma$ defined above is irreducible over $K^{nr}$.

There are several cases to consider:
\begin{itemize}
    \item if $\Delta$ is a cube in $K$, then the degree of $F/K$ is exactly $8$;
    \item if $\Delta$ is a cube in $K^{nr}$ but not in $K$, then $[L:K^{nr}]=8$ and $[F:K]=24$;
    \item if $\Delta$ is not a cube in $K^{nr}$, then $[L:K^{nr}]=[F:K]=24$.
\end{itemize}

Moreover the Galois closure of $F/K$ is given by $K(E[3])=F(\zeta_3)$, where $\zeta_3$ is a primitive $3$-rd root of unity; since if $\zeta_3 \notin K$ it generates a degree $2$ unramified extension, we have that $F/K$ is not Galois if and only if the inertia degree $n$ of $K$ over $\Q_2$ is odd. Note that this cannot occur if $\Delta$ is a cube in $K^{nr}$ but not in $K$, otherwise the extension $K(\sqrt[3]{\Delta},\zeta_3)$ would be unramified and not cyclic.

\end{section}

\begin{section}{Proof of the main theorem}\label{sec:proof}

We will use the same notation as in Section \ref{sec:good model}. Since $I$ is non-abelian, then the group $\Gal(L/K)$ is also non-abelian. By \cite[\S2 Lemma 1]{dokchitser2008root}, the representation $\rho$ factors as $\chi \otimes \psi$, where $\chi$ is the following character:
\begin{align}
    \chi : G_K &\rightarrow \overline{\Q}_\ell^\times\\
    \Frob_K & \mapsto \sqrt{-2}^n;\\
    I_K & \mapsto 1,
\end{align}
and $\psi : G_K \rightarrow \GL_2(\overline{\Q}_\ell)$ factors through the finite group  $G=\Gal(F(\zeta_3)/K)$, which is either $Q_8$ or $\SL_2(\F_3)$ if $n$ is even, $SD_{16}$ or $\GL_2(\F_3)$ if $n$ is odd. As a $G$-representation, $\psi$ is irreducible and faithful, and it is given by $\psi(g)=\dfrac{1}{\chi(g)} \rho(g)$. The definition of $\chi$ is suggested by the following lemma.

\begin{lemma}\label{eigenvalues}
Let $\Frob_F$ be the Arithmetic Frobenius of $F$; then the eigenvalues of $\rho(\Frob_F)$ are $(\pm \sqrt{-2})^{f_{F/\Q_2}}$; in particular these are real and equal if $f_{F/\Q_2}$ is even, complex conjugate if $f_{F/\Q_2}$ is odd.
\end{lemma}

\begin{proof}
Suppose that $f_{F/\Q_2}=1$. By Lemma \ref{good model} we can compute the trace of $\Frob_F$ via point-counting on the reduced curve $y^2+y=x^3$, getting
\begin{align}
    \tr(\rho(\Frob_F))= |k| + 1 - |\Tilde{E_F}(k)|=2+1-3=0.
\end{align}
Then by \cite[V \S 2 Proposition 2.3]{silverman1986arithmetic}, the characteristic polynomial of $\rho(\Frob_F)$ is
\begin{align}
    T^2 + 2,
\end{align}
with roots $\pm \sqrt{-2}$.
For general $f_{F/\Q_2}$, the eigenvalues of $\rho(\Frob_F)$ are the $f_{F/\Q_2}$-th powers of the roots of the polynomial above, hence for odd $f_{F/\Q_2}$ we get $\sqrt{-2}^{f_{F/\Q_2}}$, $-\sqrt{-2}^{f_{F/\Q_2}}$, and for even $n$ there is only one double eigenvalue $(-2)^{f_{F/\Q_2}/2}$.
\end{proof}

We have that for even $n$, $F(\zeta_3)=F$ and so $\Frob_F$ is central in the group $\Gal(L/K)$, so it acts as a scalar matrix, with eigenvalue given by Lemma \ref{eigenvalues}. Moreover, for any $n$, if $\sqrt[3]{\Delta} \notin K^{nr} \setminus K$, then $F$ and $K$ have the same residue field and so $f_{F/\Q_2}=n$; in this case $\rho(\Frob_K)=\rho(\Frob_F)$. Otherwise, the unramified part of the extension $F/K$ is given by $\sqrt[3]{\Delta}$ and therefore has degree $3$, so $f_{F/\Q_2}=3n$. In particular $\rho(\Frob_F)=\rho(\Frob_K)^3$.

Suppose first that $n$ is even and that $\sqrt[3]{\Delta} \notin K^{nr} \setminus K$. Then we have the following.

\begin{theorem}\label{neven}
If $K$ is a $2$-adic field with even inertia degree $n$ over $\Q_2$, then Theorem \ref{mainthm} is true for any elliptic curve $E/K$ with potential good reduction such that the image of inertia under $\rho$ is non-abelian and $\sqrt[3]{\Delta} \notin K^{nr} \setminus K$.
\end{theorem}

\begin{proof}
Since $n$ is even, $G$ is equal to its inertia subgroup since either $\sqrt[3]{\Delta} \in K$ or $\sqrt[3]{\Delta} \notin K^{nr}$. As noticed above, $\Frob_K$ acts as the multiplication by a scalar with eigenvalue $(-2)^{n/2}= \chi(\Frob_K)$, therefore $\psi$ is given by the representation $\rho$ restricted to inertia, hence it is a faithful, irreducible $2$-dimensional representation of $G$ (which is either $Q_8$ or $\SL_2(\F_3)$). Moreover by \cite[\S 2 Theorem 2.ii]{10.2307/1970722}, the character of this representation has values in $\Z$. By inspecting the character tables of $Q_8$ and $\SL_2(\F_3)$ on \cite{groupnames}, we deduce that each of these groups only has one such representation, the one given in the statement.
\end{proof}


For the case $\sqrt[3]{\Delta} \in K^{nr}\setminus K$, the image of inertia is strictly smaller than $\Gal(F/K)$, so the argument that the character values are in $\Z$ does not apply directly. However it is still possible to compute $\psi$, getting a result surprisingly similar to the one in Theorem \ref{neven}.

\begin{theorem}\label{unramified}
If $K$ is a $2$-adic field with even inertia degree over $\Q_2$ and $E$ is an elliptic curve with potential good reduction over $K$ such that the image of inertia under $\rho$ is non-abelian and $\sqrt[3]{\Delta} \in K^{nr} \setminus K$, then Theorem \ref{mainthm} holds for $E$.
\end{theorem}

\begin{proof}
The difference between $G$ and its inertia subgroup is determined by $\Frob_K$. We will show that the trace of $\psi(\Frob_K)$ is integer and so the result will follow from the proof of Theorem \ref{neven}.

Recall that $\chi$ is the unramified character given by $\chi(\Frob_K)=(-2)^{n/2}$; then since the inertia degree of $F/K$ is $3$ we have $\rho(\Frob_F)=\rho(\Frob_K)^3$, therefore using the relation $\rho= \chi \otimes \psi$ and the fact that $\rho(\Frob_F)$ is a scalar:
\begin{align}
    (-2)^{3n/2} id = ((-2)^{n/2} \psi(\Frob_K))^3;
\end{align}
so the eigenvalues of $\psi(\Frob_K)$ are $3$-rd roots of unity (not necessarily primitive) in $\overline{\Q}_\ell$; moreover the order of $\psi(\Frob_K)$ is exactly $3$ since $\psi$ is faithful as a representation of $G$, so not both the eigenvalues can be $1$. Computing the determinant on both sides we obtain that $\det (\psi (\Frob_K))=1$, therefore the eigenvalues of $\psi(\Frob_K)$ can only be the two distinct primitive $3$-rd roots of unity, with trace $-1$. Hence the representation $\psi$ of $\SL_2(\F_3)$ is the one given in the statement.
\end{proof}


From this moment on, we assume that $n$ is odd or equivalently that $F/K$ is not Galois. Then $\psi$ is an irreducible faithful representation of dimension $2$ of $G$, which is either $SD_{16}$ if $\Delta$ is a cube in $K$, or $\GL_2(\F_3)$ otherwise. Again by looking at the character tables of these two groups on \cite{groupnames}, we obtain two possible such representations, both of which extend the representation of inertia described in the proof of Theorem \ref{neven}. These two representations only differ for the character value on the elements of order $8$. So we need a more explicit description of the action of this group to deduce which one is the correct representation. Note that we will only concentrate on the wild subgroup of $G$, so we may assume for simplicity that the whole group is $SD_{16}$. If $G=\GL_2(\F_3)$ the wild subgroup does not change, since this Galois group differs from the previous one by a cubic totally ramified (hence tame) field extension, and the parity of $n$ is not affected.

First, we need to describe explicitly this wild group. Recall that if $\tilde E_F$ is the reduced curve of the good model for $E$ over $F$ then there is an injection of the image of inertia into $\Aut(\tilde E_F)$, that is $\SL_2(\F_3)$. This injection is obtained as follows: fix an element $\sigma$ of inertia, and a point $(\tilde x, \tilde y)$ on the reduced curve, then lift it to a point $(x,y)$ of $E_F$, which has coordinates in $F$, apply $\sigma$ to each coordinate, and then reduce to another point which again lies on $\tilde E_F$. The group $G$ contains a copy of the image of inertia and an extra element $\phi$ of order $2$; applying the same construction, we see that $\phi$ acts as Frobenius on the reduced curve. Now fix $\ell =3$ and consider the representation $\overline{\rho}$ which is the $3$-adic representation modulo $3$. This is the Galois representation induced by $\rho$ on $E[3]$; after fixing a basis $\{P,Q\}$ for $E[3]$ as a $\F_3$-vector space, $\overline{\rho}$ takes values in $\GL_2(\F_3)$. In \cite[\S 4 Figure 4.2]{Robson2017} there is a visual interpretation of this action. Note that the two representations described above are identical, since they are both induced by the action of the Galois group on elements of $F(\zeta_3)$. We will use both interpretations to find the character of the generators of the group $G$ under $\psi$.


\begin{lemma}\label{frobmod3}
There exists a basis $\{P,Q\}$ of $E[3]$ where the matrix representing the image of Frobenius modulo $3$ is $\mat{1}{0}{0}{2}$.
\end{lemma}

\begin{proof}
Let $P$ be as in the proof of Lemma \ref{good model}. Then $P$ and $-P$ are the only points of exact order $3$ with coordinates in $F$. Otherwise if $Q=(x_Q,y_Q)$ is another point of order $3$ and $x_Q,y_Q \in F$, then the coordinates of every other point of order $3$ would be rational functions with rational coefficients of $x_P,y_P,x_Q,y_Q$ since $E[3]=\{ O, \pm P, \pm Q, \pm P \pm Q \}$, hence these coordinates would be in $F$, thus $F/K$ would be Galois, contradiction.

We know that the good model for $E_F$ reduces to $y^2+y=x^3$, and by direct computation this curve has the following $8$ points of exact order $3$:
\begin{align}
     (0,0), (0,1), (\zzeta,\zzeta),(\zzeta,\zzeta^2), (1,\zzeta),(1,\zzeta^2) \text{ and } (\zzeta^2,\zzeta),(\zzeta^2,\zzeta^2), 
\end{align}
where $\zzeta$ is a third root of unity in $\overline{K}$.

After applying the change of coordinates described in Lemma \ref{good model}, $P$ reduces to $(0,0)$ and $-P$ to $(0,1)$. Let $Q$ be the $3$-torsion point of $E(F(\zeta_3))$ reducing to $(1,\zzeta)$. Then under $\overline{\rho}$, Frobenius acts trivially on $P$ and maps $Q$ to $-Q$, that is to the only point that has the same abscissa of $Q$, which is in $F$. Therefore if we complete $P$ to the basis $\{P,Q\}$ of $E[3]$ with $Q$ as above, the matrix expressing the Frobenius in this basis is $\mat{1}{0}{0}{2}$, as claimed.
\end{proof}

\begin{theorem}\label{nodd}
If $K$ is a $2$-adic field with odd inertia degree $n$ over $\Q_2$, then Theorem \ref{mainthm} is true for any elliptic curve $E/K$ with potential good reduction such that the image of inertia under $\rho$ is non-abelian.
\end{theorem}

\begin{proof}
We will denote by $b$ the matrix $\overline{\rho}(\phi) = \mat{1}{0}{0}{2} \in \GL_2(\F_3)$. Now let us choose an element $\sigma$ in the inertia subgroup, of order $4$, for example
\begin{align}
    P & \mapsto Q-P\\
    Q & \mapsto P+Q.
\end{align}

It exists since $Q_8$ is contained in the image of inertia under $\overline{\rho}$, therefore every element of $\GL_2(\F_3)$ with determinant $1$ and $2$-power order is in the image of inertia. Then $\overline{\rho}(\sigma) $ is given by the matrix $\mat{2}{1}{1}{1}$. The element $\phi \sigma$ is an element of order $8$ of the group $G$ and so if we determine $\tr \psi (\phi \sigma)$, we determine the irreducible representation $\psi$. To compute this trace, we look at the trace of $\rho(\Frob_K \sigma)$. Let $a$ be the reduction of $\rho(\Frob_K \sigma )$ modulo $3$. Then $a=\mat{2}{1}{2}{2}$, with trace $1$. This means that $\tr (\rho (\Frob_K \sigma)) \equiv 1 \pmod 3$. Note that $a,b$, with the relations $a^8 = b^2 =1, bab=a^3$ generate $SD_{16}$ as a subgroup of $\GL_2(\F_3)$ (see the presentation of $SD_{16}$ in \cite{groupnames}).

By looking at the character table of the group $SD_{16}$ in \cite{groupnames}, we deduce that $\tr \psi(\phi \sigma)$ is either $ \sqrt{-2}$ or $-\sqrt{-2}$, so
\begin{align}
\tr \rho( \Frob_K \sigma )= \chi(\Frob_K) \tr (\psi( \phi \sigma)) \in \{\sqrt{-2}^n \cdot (\pm \sqrt{-2})\}.
\end{align}

Only one of this two numbers is congruent to $1$ modulo $3$, namely the one we obtain if $\tr \psi (\phi \sigma )=+ \sqrt{-2}$. Therefore we have the following character for $\psi$ (note that only the generators of the conjugacy classes of elements outside inertia, which identify the correct representation, are explicitly written).
$$
\begin{array}{c|rrrrrrr}
  \rm class&\rm1&\rm2A&\rm2B&\rm4A&\rm4B&\rm8A&\rm8B\cr
  \rm size&1&1&4&2&4&2&2\cr
  \rm generator& & &\phi& & &\phi \sigma&\phi \sigma^{-1}\cr
\hline
  \psi&2&-2&0&0&0&\sqrt{-2}&-\sqrt{-2}\cr
\end{array}
$$

Similarly if the inertia image is $\SL_2(\F_3)$, we get the following character for $\psi$:
$$
\begin{array}{c|rrrrrrrr}
  \rm class&\rm1&\rm2A&\rm2B&\rm3&\rm4&\rm6&\rm8A&\rm8B\cr
  \rm size&1&1&12&8&6&8&6&6\cr
  \rm generator&&&\phi&&&&\phi \sigma&\phi \sigma^{-1}\cr
\hline
  \psi&2&-2&0&-1&0&1&\sqrt{-2}&-\sqrt{-2}\cr
\end{array}
$$
as stated.
\end{proof}
\end{section}

\begin{section}{Notes on the implementation}
As explained in \cite{1812.05651}, the representation described here is the dual of the representation on the \'etale cohomology of $E$. In particular the function GaloisRepresentation implemented in MAGMA computes the Galois representation on the \'etale cohomology. Concretely, the two only differ by the character value of $\psi$ on the elements of the two conjugacy classes $8A$ and $8B$. The function is currently being improved implementing the result presented here.

Theorem \ref{mainthm} and Theorems 3.1 and 3.2 of \cite{1812.05651} give a method to describe completely the $\ell$-adic Galois representation of an elliptic curve with potential good reduction and non-abelian inertia action.
\end{section}

\section*{Acknowledgments}

The author thanks her supervisor Tim Dokchitser for the useful conversations and corrections. This work was supported by EPSRC.


\end{document}